\documentclass{article}

\usepackage{amsmath}
\usepackage{amssymb}
\usepackage{amsthm}
\usepackage{ifthen}

\newcommand{\N}{{\mathbb N}}
\newcommand{\C}{{\mathbb C}}
\newcommand{\R}{{\mathbb R}}
\newcommand{\p}{\partial}

\newcommand{\Lam}{\Lambda}

\newcommand{\parder}[3][Default]{
	\frac{\partial \ifthenelse{\equal{#1}{Default}}{}{^{#1}}#2}{
              \partial #3 \ifthenelse{\equal{#1}{Default}}{}{^{#1}}}}

\newtheorem{theorem}{Theorem}

\newtheorem{corollary}[theorem]{Corollary}
\newtheorem*{gvc}{Generalized Vanishing Conjecture (GVC)}

\begin{document}

\title{A few remarks on the Generalized Vanishing Conjecture}

\author{Michiel de Bondt\footnote{The author was supported by the Netherlands 
                          Organisation for Scientific Research (NWO)} \\
Department of Mathematics, Radboud University \\
Postbus 9010, 6500 GL Nijmegen, The Netherlands \\
\emph{E-mail:} M.deBondt@math.ru.nl}

\maketitle

\begin{abstract}
\noindent
We show that the Generalized Vanishing Conjecture $$\forall_{m \ge 1} [\Lam^m f^m = 0] \Longrightarrow
\forall_{m \gg 0} [\Lam^m (g f^m) = 0]$$ for a fixed differential operator $\Lam \in 
k[\partial]$ follows from a special case of it, namely that the additional factor $g$ is a power of the 
radical polynomial $f$. Next we show that in order to prove the Generalized Vanishing Conjecture
(up to some bound on the degree of $\Lam$), we may assume that $\Lam$ is a linear combination
of powers of distinct partial derivatives. At last, we show that the 
Generalized Vanishing Conjecture holds for products of linear forms in $\partial$, in particular 
homogeneous differential operators 
$\Lambda \in k[\partial_1,\partial_2]$.
\end{abstract}



\section*{Introduction}

The Jacobian Conjecture has been the topic of many papers (see \cite{MR0663785} 
and \cite{MR1790619} and its references). Until recently, there was no framework 
available in which this notorious conjecture could be studied.
Based on work in \cite{MR2138860}, Wenhua Zhao published several papers 
(\cite{MR2134306}, \cite{MR2247890}, \cite{MR2368014}, \cite{MR2859886})
which have changed this situation dramatically.

In these papers, he introduced various conjectures which imply the Jacobian
Conjecture. One of these conjectures is the so-called Generalized Vanishing 
Conjecture. To describe it, we fix the following notations. Let 
$k[x]=k[x_1,\ldots,x_n]$ be the polynomial ring in $n$ variables over a field 
$k$. By $D=k[\p_1,\ldots,\p_n]$ we denote the ring of differential operators 
with constant coefficients. By $k$-linearity of taking partial derivative,
the following defines $\Lam f \in k[x]$ with $\Lam \in D$ and $f \in k[x]$ 
uniquely:
\begin{align*}
(\Lam_1+\Lam_2) f &= \Lam_1 f + \Lam_2 f & 
(\Lam_1\Lam_2) f &= \Lam_1(\Lam_2 f) & \p_i f = \parder{}{x_i} f\mbox{,}
\end{align*}
where $\Lam_1, \Lam_2 \in D$ and $f \in k[x]$.

\begin{gvc} 
Let $\Lam\in D$ and $f\in k[x]$ such that 
\begin{align*}
\Lam^mf^m&=0 \mbox{ for all }m\geq 1\mbox{.}
\intertext{Then for all $g\in k[x]$, we have}
\Lam^m(gf^m)&=0\mbox{ for all }m \gg 0\mbox{.}
\end{align*}
\end{gvc}

\noindent
It was shown in \cite[Th.\@ 7.2]{MR2247890} that for a field $k$ of characteristic zero,
a positive answer to this conjecture (in all 
dimensions), with $\Lam$ being the Laplace operator $\Delta$ (and $g = f$), implies the Jacobian 
Conjecture. For a field of positive characteristic $p$, the GVC can easily be proved, because
$\Lam^p g = 0$ for all $\Lam \in k[\p]$ with trivial constant part and all $g \in k[x]$.

The main results of this paper can be described as follows. First we show 
that the $g$'s in the formulation of the GVC can be replaced by powers of $f$. 
We will do that in a corollary of the following theorem.

\begin{theorem} \label{thm1}
Let $\tilde{f},g \in k[x]$ and $m \ge d$. Suppose that 
\begin{align*}
\Lam^{m-d} \tilde{f} = 0
\intertext{for some $\Lam \in D$. If $\deg g \le d$, then}
\Lam^m (g \tilde{f}) = 0
\end{align*}
as well.
\end{theorem}

\begin{corollary} \label{col1}
The GVC (for some $\Lam \in D$) is equivalent to the following statement: 
if $f\in k[x]$ is such that 
\begin{align*}
\Lam^m f^m&=0 \mbox{ for all }m\geq 1\mbox{,}
\intertext{then for each $d\geq 1$, we have}
\Lam^m f^{m+d}&=0\mbox{ for all }m \gg 0\mbox{.}
\end{align*}
\end{corollary}

\begin{proof}
The statement of corollary \ref{col1} follows from the GVC (for $\Lam \in D$) by taking $g = f^d$, 
so it remains to prove the converse. For that purpose, let $g \in k[x]$ and choose $d \ge \deg g$.
Combining the condition $\Lam^mf^m =0$ for all $m\geq 1$ of the GVC (for $\Lam \in D$) and the 
statement of corollary \ref{col1}, we get $\Lam^mf^{m+d} = 0$ for all $m \gg 0$, which is 
equivalent to $\Lam^{m-d}f^m=0$ for all $m \gg 0$. By taking $\tilde{f} = f^m$ in
theorem \ref{thm1}, we subsequently obtain $\Lam^m (g f^m) = 0$ for all $m \gg 0$.
\end{proof}

\noindent
In the proof of \cite[Th.\@ 1.5]{MR2418165}, corollary \ref{col1} is proved for 
$\Lam = \Delta$, the Laplace operator. The claim of \cite[Th.\@ 1.5]{MR2418165} is that
one can even take $d = 1$ in corollary \ref{col1} when $\Lam = \Delta$, which 
subsequently follows from (3) $\Rightarrow$ (2) of \cite[Th.\@ 6.2]{MR2247890}.
Hence we can take $g = f$ in the GVC when we restrict ourselves to $\Lam = \Delta$.

If $k$ has characteristic zero, then by \cite[\S 1.1]{MR1356713}, we can also formulate
theorem \ref{thm1} and corollary \ref{col1} in terms of the $n$-th Weyl algebra $A_n(k)$ over $k$.
$A_n(k)$ is the algebra of skew polynomials in $x$ and $\p$ over $k$, with the following commutator 
relations, where $\delta$ is Kronecker's delta:
\begin{align*}
x_ix_j-x_jx_i&=\p_i\p_j-\p_j\p_i=0 & \p_ix_j-x_j\p_i=\delta_{ij} & \mbox{ for all } i,j\mbox{.}
\end{align*}
We get the Weyl algebra formulation of theorem \ref{thm1} and corollary \ref{col1} as follows.
In each of the equalities of the form $E = 0$ in them, we interpret the left hand side $E$
as an element of $A_n(k)$, and replace `$E = 0$' by
`$E$ is contained in the left ideal of $A_n(k)$ generated by $\p_1, \p_2, \ldots, \p_n$'. See
\cite[\S 5.1]{MR1356713} for the justification of this reformulation, especially 
\cite[Prop.\@ 5.1.2]{MR1356713}.

If $k$ has positive characteristic, then the Weyl algebra formulations of theorem \ref{thm1} and 
corollary \ref{col1} are not just reformulations of theorem \ref{thm1} and corollary \ref{col1},
but really different claims, see \cite[\S 2.3]{MR1356713}. This is however not essential for the 
proof of corollary \ref{col1}, and neither will be essential for the proof of theorem \ref{thm1} 
in the next section, so we actually have two theorems \ref{thm1} and corollaries \ref{col1}. 

Next we show that it suffices to investigate the GVC for a special class of operators. 
More precisely we show the following.

\begin{theorem} \label{thm2}
It suffices to investigate the GVC in all 
dimensions for operators of the form 
$$
c_1\p_1^{d_1}+c_2\p_2^{d_2}+\cdots+c_n\p_n^{d_n}\mbox{,}
$$
where the $d_i$ are positive integers and $c_i \in k$. Furthermore, we may take 
$c_i = 1$ or $c_i^2 = 1$ for all $i$ when $k = \C$ or $k = \R$ respectively. 
\end{theorem}

\noindent 
Finally we prove the GVC for the following class of operators.

\begin{theorem} \label{thm3}
Let $\Lam\in D$ be a product of linear
forms (in the $\p_i$). Then the GVC holds for $\Lam$.
\end{theorem}

\noindent 
As a consequence, we can deduce the following.

\begin{corollary} In dimension two, the GVC holds for any homogeneous operator.
\end{corollary}

\begin{proof}
Assume without loss of generality that $k$ is algebraically closed. Then
any homogeneous polynomial in two variables is a product
of linear factors. Hence the result follows from theorem \ref{thm3}.
\end{proof}

\section*{The proofs}

Fix $\tilde{f} \in k[x]$. For $\Lam \in D$ and $f,g \in k[x]$, define
$$
[\Lam,g]f := \Lam(gf)-g(\Lam f)\mbox{.}
$$
Notice that by the product rule of diffentiation,
\begin{equation} \label{Leib}
[\p_i,g]f = \p_i(gf)-g(\p_if) = g_{x_i} f 
\end{equation}
for all $f,g \in k[x]$, where $g_{x_i}$ is the polynomial $\parder{}{x_i} g$.
More generally, write $\p^\alpha = \p_1^{\alpha_1} \p_2^{\alpha_2} \cdots 
\p_n^{\alpha_n}$. If $\alpha_i = 0$ for all $i$, then trivially 
$[\p^{\alpha},g] \tilde{f} = 0$ for all $g \in k[x]$. Hence assume that $\alpha_i \ne 0$
for some $i$ and define 
$\p^{\hat{\alpha}}$ by $\p^{\hat{\alpha}} \p_i = \p^{\alpha}$. Using
\eqref{Leib} with $f = \p^{\hat{\alpha}} \tilde{f}$, we obtain
\begin{align*}
[\p^{\alpha},g] \tilde{f} 
&= \p^{\alpha}(g \tilde{f}) - g (\p^{\alpha} \tilde{f}) \\
&= \p_i\big(\p^{\hat{\alpha}} (g \tilde{f}) - g (\p^{\hat{\alpha}} \tilde{f})\big) 
   + \p_i\big(g (\p^{\hat{\alpha}} \tilde{f})\big) - g \big(\p_i(\p^{\hat{\alpha}} \tilde{f})\big) \\
&= \p_i\big([\p^{\hat{\alpha}},g]\tilde{f}\big) + [\p_i,g] (\p^{\hat{\alpha}} \tilde{f}) \\
&= \p_i\big([\p^{\hat{\alpha}},g]\tilde{f}\big) - -g_{x_i} (\p^{\hat{\alpha}} \tilde{f}) \\
&= \p_i\big([\p^{\hat{\alpha}},g]\tilde{f}\big) + \p^{\hat{\alpha}}(g_{x_i} \tilde{f}) 
   - [\p^{\hat{\alpha}},g_{x_i}] \tilde{f}
\end{align*}
for all $g \in k[x]$. By induction on $\sum_{i=1}^n \alpha_i$, it follows that $[\p^{\alpha},g] 
\tilde{f}$ can be expressed as a $k$-linear combination of polynomials of the form 
$\Lam^{*} (g^{*} \tilde{f})$, where $\Lam^{*} \in D$ and $g^{*} \in k[x]$ is a polynomial of degree 
less than $\deg g$.

Now fix $g \in k[x]$ as well and set $d := \deg g$.
If we define $D_{[x]}^{(r)} \tilde{f}$ as the $k$-space of polynomials $\Lam^{*} (g^{*} \tilde{f})$, 
with $\Lam^{*} \in D$ and $g^{*} \in k[x]$ of degree $\le r$, then
$[\p^{\alpha},g]\tilde{f} \in D_{[x]}^{(d - 1)}\tilde{f}$. 
Writing an arbitrary operator $\Lam \in D$ as a $k$-linear combination of monomials $\p^{\alpha}$,
we deduce that 
\begin{equation} \label{eqn1}
[\Lam,g]\tilde{f} \in D_{[x]}^{(d - 1)}\tilde{f} \mbox{ for all $\Lam \in D$.}
\end{equation}

\begin{proof}[Proof of theorem \ref{thm1}]
Assume $\deg g \le d$. Recall that we have $\Lam^{m-d} \tilde{f} = 0$ and
must show that the left hand side $\Lam^m (g \tilde{f})$ of (\ref{eqn2}) below is zero. 
If $m=d$, then $\tilde{f}= 0$. If $d = 0$, then $g \in k$. Hence the cases $m = d$ and $d = 0$ 
are trivial. So assume that $m > d > 0$. Notice that $\Lam (g\tilde{f}) = g (\Lam \tilde{f}) + 
[\Lam,g]\tilde{f}$, whence
\begin{equation}
\Lam^m (g \tilde{f}) = \Lam^{m-1} \big(g (\Lam \tilde{f})\big)
+ \Lam^{m-1} \big([\Lam, g] \tilde{f}\big)\mbox{.} \label{eqn2}
\end{equation}
Since $\Lam^{(m-1)-d} (\Lam \tilde{f}) = \Lam^{m-d} (\tilde{f}) = 0$, it follows 
by induction on $m$ that the first term $\Lam^{m-1} \big(g (\Lam \tilde{f})\big) = 0$ on 
the right hand side of (\ref{eqn2}) vanishes. 

Hence it suffices to show that the second term $\Lam^{m-1} \big([\Lam, g] \tilde{f}\big)$ on 
the right hand side of (\ref{eqn2}) vanishes as well. For that purpose, we use \eqref{eqn1}
to write $[\Lam, g] \tilde{f}$ as a sum of polynomials of the form $\Lam^{*} (g^{*} \tilde{f})$, 
where $\Lam^{*} \in D$ and $g^{*} \in k[x]$ such that $\deg g^{*} \le d-1$ for each such term. 
Thus the second term $\Lam^{m-1} \big([\Lam, g] \tilde{f}\big)$ on the right hand side of 
(\ref{eqn2}) is a sum of polynomials of the form
\begin{equation}
\Lam^{m-1} \Lam^{*} (g^{*} \tilde{f}) = \Lam^{*} \big(\Lam^{m-1} (g^{*} \tilde{f})\big) \label{eqn3}\mbox{,}
\end{equation}
with $\Lam^{*}$ and $g^{*}$ as above.
Since $\Lam^{(m-1)-(d-1)} (\tilde{f}) = \Lam^{m-d} (\tilde{f}) = 0$, it follows 
by induction on $d$ that the second factor on the right hand side of (\ref{eqn3})
vanishes. So both sides of (\ref{eqn2}) are zero.
\end{proof}

\noindent
Just as theorem \ref{thm1} itself, the above proof can be reformulated in terms of the $n$-th
Weyl algebra $A_n(k)$ over $k$ (resulting in the proof of a different theorem when $k$ has positive
characteristic). The Weyl algebra formulation of \eqref{Leib} is \cite[Exrc.\@ 1.4.1]{MR1356713}
and the definition of $D_{[x]}^{(r)} \tilde{f}$ corresponds to a filtration which is isomorphic
to the order filtration in \cite[\S 7.2]{MR1356713}, see also \cite[\S 3.2]{MR1356713}.
The actual isomorphism is generated by 
\begin{align*}
x_i &\mapsto \p_i & \p_i &\mapsto -x_i & \mbox{ for all $i$.}
\end{align*} 
By way of the same 
isomorphism, it follows from \cite[Prop.\@ 1.2.1]{MR1356713} that every element of $A_n(k)$ is a 
linear combination of elements of the form $\Lam^{*} g^{*}$, where $\Lam^{*} \in D$ and 
$g^{*} \in k[x]$. This fact has some connection with \eqref{eqn1}.

\begin{proof}[Proof of theorem \ref{thm2}.] Let $\Lam\in D$ be a non-zero operator and 
$f\in k[x]$ such that $\Lam^m f^m=0$ for all $m\geq 1$. We must show that
$\Lam^m (g f^m)=0$ in case the GVC holds for $k$-linear combinations of powers of distinct
partial derivatives.
\begin{enumerate}

\item[i)] Since each monomial of degree $d$ 
appearing in $\Lam$ can be written as a $k$-linear combination of powers 
of the form $l^d$, where $l$ is a $k$-linear combination of $\p_1, \p_2, \ldots, \p_n$,
we can write $\Lam$ as a $k$-linear combination of such powers, i.e.
\begin{equation} \label{eq3}
\Lam=c_1l_1^{d_1}+c_2l_2^{d_2}+\cdots+c_Nl_N^{d_N}
\end{equation}
for some $c_i\in k^{*}$, $d_i\in\N$, where 
$l_i=a_{i1}\p_1+\cdots+a_{in}\p_n$, for some $a_{ij}\in k$.
If $k= \C$ or $k = \R$, then we may assume that $c_1 = c_2 = \cdots = c_n = 1$
or $c_1^2 = c_2^2 = \cdots = c_n^2 = 1$ respectively. 

\item[ii)] Now we introduce $N$ new variables $y_1,y_2,\ldots,y_N$ and 
consider the operator
\begin{equation} \label{eq4}
\Lam^*:=(\p_{y_1}+l_1)^{d_1}+(\p_{y_2}+l_2)^{d_2}+\cdots+(\p_{y_N}+l_N)^{d_N}
\end{equation}
on the polynomial ring $k[x_1,\ldots,x_n,y_1,\ldots,y_N]$. 
Making the linear coordinate change defined by
\begin{align*}
x_i'&:=x_i-(a_{1i}y_1+\cdots+a_{Ni}y_N) & 
y_j'&:=y_j
\end{align*}
for all $i \le n$ and for all $j \le N$, we obtain $\p_{y_j}+l_j = \p_{y_j'}$, because
\begin{align*}
(\p_{y_j}+l_j) x_i' &= -a_{ji} + a_{ji} = 0  & 
(\p_{y_j}+l_j) y_t' &= \delta_{jt}
\end{align*}
for all $i \le n$ and all $j,t \le N$. Hence on these new coordinates, the 
operator $\Lam^*$ is of the form as described in the statement of the theorem.
It follows that the GVC holds for $\Lam^{*}$.

\item[iii)] Since $f\in k[x]$ satisfies $\Lam^mf^m=0$ for all 
$m\geq 1$, it follows from (\ref{eq3}) and (\ref{eq4}) that also $(\Lam^{*})^m f^m=0$ for 
all $m\geq 1$. As observed in ii), the GVC holds for $\Lam^*$, thus we obtain 
that for every $g\in k[x]$ also $(\Lam^{*})^m(gf^m)=0$ for all large $m$. But
again by (\ref{eq3}) and (\ref{eq4}), it follows that $\Lam^m(gf^m)=0$ for all large $m$, 
which concludes the proof. \qedhere

\end{enumerate}
\end{proof}

\noindent 
Finally, we give the proof of theorem \ref{thm3}, in which we re-use some 
techniques given in the proof of theorem \ref{thm2}.

\begin{proof}[Proof of theorem \ref{thm3}.] Let $\Lam=l_1\cdot l_2\cdot \cdots \cdot l_N$ 
for some non-zero linear forms $l_i$ in $\p_1, \p_2, \ldots, \p_n$.
As above, introduce $N$ new variables $y_1,y_2\ldots,y_N$ and consider the 
operator
\begin{equation} \label{eq5}
\Lam^{*}:=(\p_{y_1}+l_1)\cdot (\p_{y_2}+l_2)\cdot\cdots\cdot(\p_{y_N}+l_N)
\end{equation}
Making the linear coordinate change 
\begin{align*}
x_i'&:=x_i-(a_{1i}y_1+\cdots+a_{Ni}y_N) &
y_j'&:=y_j
\end{align*}
for all $i \le n$ and for all $j \le N$ again, we obtain that on these new coordinates,
the operator $\Lam^{*}$ is of the form
$$
\p_{y_1'}\cdot\p_{y_2'}\cdot\cdots\cdot\p_{y_N'}\mbox{.}
$$
By \cite[Cor.\@ 3.5.3]{0903.1478}, it follows that the GVC holds for this 
operator. Hence the GVC holds for $\Lam^{*}$. So just as in iii) above, we 
deduce that the GVC holds for $\Lam$.
\end{proof}

\bibliographystyle{myalphanum}
\bibliography{references}

\end{document}